\documentclass{amsart}
\usepackage{amsthm}
\usepackage{amsmath}
\usepackage{textcomp}
\usepackage{pb-diagram}
\usepackage[cp1250]{inputenc}
\usepackage[T1]{fontenc}
\usepackage{amssymb,graphicx}
\usepackage{amsmath}
\usepackage{pxfonts}
\usepackage{latexsym}
\usepackage{textcomp}
\usepackage{amsfonts}
\usepackage{tipa}
\usepackage{dsfont}
\usepackage{stmaryrd}
\usepackage{txfonts}
\usepackage{pifont}
\usepackage{geometry}
\usepackage{fleqn}
\usepackage{hyperref}
\DeclareSymbolFont{script}{U}{eus}{m}{n}
\DeclareMathSymbol{\Wedge}{0}{script}{"5E}

\newcommand{\mult}{^{\scriptscriptstyle\times}}

\newtheorem{definition}{Definition}
\newtheorem{theorem}{Theorem}
\newtheorem{lemma}{Lemma}

\newtheorem{proposition}{Proposition}
\newtheorem{wniosek}{Corollary}
\theoremstyle{definition}
\newtheorem{remark}{Remark}
\newtheorem{example}{Example}

\newtheorem{observation}{Observation}

\DeclareMathOperator{\id}{id}

\DeclareMathOperator{\SO}{SO}
\DeclareMathOperator{\SL}{SL}

\DeclareMathOperator{\F}{F}

\DeclareMathOperator{\constans}{const}
\DeclareMathOperator{\de}{d}

\DeclareMathOperator{\spanu}{span}

\DeclareMathOperator{\Sgp}{\Sigma^+}
\DeclareMathOperator{\Sgm}{\Sigma^-}
\DeclareMathOperator{\Lap}{\Lambda^+}
\DeclareMathOperator{\Lam}{\Lambda^-}
\DeclareMathOperator{\ot}{\otimes}

\title{Twistor construction of asymptotically hyperbolic Einstein--Weyl spaces}
\author{Aleksandra Bor\'owka}
\email{aleksandra.borowka@uj.edu.pl}
\address{Jagiellonian University in Krakow, Institute of Mathematics, ul. prof. Stanis\l{}awa \L{}ojasiewicza 6,
30-348 Krak\'ow, +48126647627 }
\begin{document}

\begin{abstract}
Starting from a real analytic conformal Cartan connection on a real analytic surface $S$, we construct a complex surface $T$ containing a family of pairs of projective lines. Using the structure on $S$ we also construct a complex $3$-space $Z$, such that $Z$ is a twistor space of a self-dual conformal $4$-fold and $T$ is a quotient of $Z$ by a holomorphic local $\mathbb{C}^{\mult}$ action. We prove that $T$ is a minitwistor space of an asymptotically hyperbolic Einstein-Weyl space with $S$ as an asymptotic boundary.
\end{abstract}
\maketitle
\section{Introduction}

An Einstein-Weyl manifold $(B,c,\mathcal{D})$ is a conformal manifold $(B,c)$ whose symmetric trace-free part of the Ricci tensor of $\mathcal{D}$ vanishes (see for example \cite{CaPe1}). N. Hitchin (\cite{Hit}) developed minitwistor theory for $3$-dimensional Einstein-Weyl manifolds and P. Jones and K. Tod (\cite{JT}) showed that quotients of twistor spaces of self-dual conformal $4$-manifolds by a well-behaving holomorphic actions are minitwistor spaces.

C. LeBrun (\cite{Le}) stated a definition of asymptotically hyperbolic Einstein-Weyl spaces and, using similar techniques as P. Jones and K. Tod, he showed that the quotient of a self-dual conformal $4$-manifold by a well-behaving $S^1$ action is an asymptotically hyperbolic Einstein-Weyl manifold. He also stated a question if all asymptotically hyperbolic Einstein-Weyl manifolds (in the global sense) are in fact hyperbolic. 

The main result of this paper is a construction, which from a real analytic conformal Cartan connection on a real analytic surface $S$ produces a twistor space $T$ of an asymptotically hyperbolic Einstein-Weyl manifold $B$ with $S$ as an asymptotic boundary. We also give a natural construction of a twistor space $Z$ of a self-dual conformal $4$-manifold $M$ such that $B$ is the quotient of $M$ by a conformal semi-free $S^1$ action.

Section \ref{s1} contains necessary background. Firstly, we introduce twistor theory for self-dual conformal $4$-manifolds and Einstein-Weyl $3$-manifolds. Then we discuss briefly asymptotically hyperbolic Einstein-Weyl spaces and conformal Cartan connections. Finally we recall properties of complexifications. 

In Section \ref{s2} we discuss a motivating example for the construction, namely the quadric in $\mathbb{CP}^3$. 
In Section~\ref{s2}, starting from a complexification $\Sigma$ of a real-analytic surface $S$ with a real-analytic Cartan geometry, we construct a complex surface $T$ containing family of line pairs parametrised by points in $\Sigma$. 
In Section \ref{s3}, using the same data, we give a natural construction of a $3$-dimensional complex manifold $Z$ such that there is a projection $\Pi$ from an open subset of $Z$ to $T$. 

In Section \ref{s4}, we prove that the manifold $Z$ admits a family of projective lines with normal bundle $\mathcal{O}(1)\oplus\mathcal{O}(1)$. We also show that $Z$ and $T$ admit compatible real structures and hence conclude that $Z$ is a twistor space of a self-dual conformal $4$-manifold $M$. Next, we show that the projection from $Z$ to $T$ is given by a holomorphic $\mathbb{C}^{\mult}$ action which induces a semi-free $S^1$ action on $M$. We prove that the quotient of some twistor lines gives minitwistor lines on $T$ and therefore conclude that $T$ is a minitwistor space of an Einstein-Weyl manifold $B$. Applying a result of LeBrun (\cite{Le}), we conclude that $B$ is asymptotically hyperbolic.

As we use twistor methods for our construction, the results we obtain are local in nature. In particular the constructed manifolds are not counterexamples for the LeBrun's conjecture.

\section{Background} \label{s1}
\subsection*{Twistor theory}
Twistor theory introduced by R. Penrose in 1965 (\cite{PR}) provides a correspondence between various classes of real smooth manifolds (e.g. self-dual conformal $4$-manifolds (\cite{AtHS}), Einstein-Weyl $3$-manifolds (\cite{Hit}), surfaces with projective structure \cite{HitP}) with complex manifolds containing special families of projective lines. 

Recall that $\mathcal{O}(-1)$ is the tautological bundle over $\mathbb{CP}^n$, $\mathcal{O}(1)=\mathcal{O}(-1)^*$ and $\mathcal{O}(k)=\mathcal{O}(1)^{\otimes k}$. We will use the following two theorems.
\begin{theorem}[Penrose \cite{PR}, Atiyah, Hitchin, Singer \cite{AtHS}]\label{TPenrose}
Let $Z$ be a complex $3$-manifold such that:
\begin{itemize}
\item[(i)]There is a family of non-singular holomorphic projective lines $\mathbb{C}\mathbb{P}^1$ each with normal bundle isomorphic to $\mathcal{O}(1)\oplus\mathcal{O}(1)$,
\item[(ii)]$Z$ has a real structure which on lines from the family which are invariant under this real structure induces the antipodal map of $\mathbb{C}\mathbb{P}^1$.
\end{itemize}
Then the parameter space of projective lines invariant under the real structure is a self-dual conformal $4$-manifold for which $Z$ is the twistor space.
\end{theorem}

\begin{theorem}[Hitchin, \cite{Hit}]\label{THitchin}
Let $T$ be a surface such that:
\begin{itemize}
\item[(i)]There is a family of non-singular holomorphic projective lines $\mathbb{C}\mathbb{P}^1$ each with normal bundle isomorphic to $\mathcal{O}(2)$,
\item[(ii)]$T$ has a real structure which on lines from the family which are invariant under this real structure induces the antipodal map of $\mathbb{C}\mathbb{P}^1$.
\end{itemize}
Then the parameter space of projective lines invariant under the real structure is an Einstein--Weyl manifold.
\end{theorem}

These two twistor-type construction have been connected by P. Jones and K. Tod (\cite{JT}). They have shown that the quotient of any twistor space $Z$ by a well-behaving holomorphic action is a minitwistor space $T$. Well-behaving holomorphic actions on twistor spaces correspond to conformal actions on the corresponding self-dual conformal $4$ manifolds and the Einstein-Weyl space arising from $T$ is the quotient of self-dual conformal $4$-manifold arising from $Z$ by the conformal action.

\subsection*{Asymptotically hyperbolic Einstein-Weyl spaces}
Recall that hyperbolic $3$-manifolds are Riemannian $3$-manifolds of constant sectional curvature equal to $-1$. They are basic examples of Einstein and thus Einstein--Weyl manifolds. C. LeBrun in \cite{Le} stated the definition which generalises this notion. 
\begin{definition}[\cite{Le}]\label{defashyp}
An Einstein--Weyl $3$-manifold $(B,c,\mathcal{D})$ is said to be asymptotically hyperbolic if there exist:
\begin{itemize}
\item a connected Riemannian 3-manifold $(Y,\tilde{h})$ with a boundary $\partial Y$,
\item a defining function of $\partial Y$, i.e., a function $f:Y\rightarrow \mathbb{R}^+\cup\{0\}$ such that $f|_{\partial Y}=0$ and $(\de f)|_{\partial Y}$ is non-vanishing,
\item a smooth $1$-form $\omega$ on $Y$ vanishing along $\partial Y$,
\end{itemize}
such that $(B,c,\mathcal{D})$ is isomorphic to $(Y-\partial Y)$ equipped with the Weyl structure given by the conformal class of $[\tilde{h}]$ and the Weyl connection given by $(f^{-2}\tilde{h},\omega)$.
\end{definition}
Recall that a semi-free action of a group $G$ is such that the stabiliser of every point is either trivial or the whole $G$. Later we will use the following lemma.

\begin{lemma}[\cite{Le}]\label{lemmaLe}
Suppose that $(M,g)$ is a self-dual manifold with an $S^1$ action which has a non-empty surface $S$ of fixed points, is semi-free and does not have isolated fixed points. Let $B$ be a maximal smooth manifold (without boundary) contained in $Y=M/S^1$ i.e., the smooth point subset of $Y$. Then the Einstein--Weyl structure defined by the Jones--Tod correspondence (\cite{JT}) on $B$ is an asymptotically hyperbolic Einstein--Weyl structure.
 
\end{lemma}

In \cite{Le} C. LeBrun stated hypothesis that if $Y$ from Definition \ref{defashyp} is compact then the asymptotically hyperbolic manifold $B$ is hyperbolic.

\subsection*{Conformal Cartan connection}
Conformal Cartan geometry (like other parabolic geometries; see \cite{CSch}, \cite{CS2}) can be defined using tractor bundles (see \cite{Est}, \cite{CaBu2}). We will use this approach in the paper.

\begin{definition} \label{CartanCon}
A conformal Cartan geometry on an $n$-manifold $S$ is a quadruple $(V,\langle,\rangle,\Lambda,\mathcal{D})$ where:
\begin{itemize}
\item $V$ is a rank $n+2$ vector bundle with inner product $\langle,\rangle$ over $S$,
\item $\Lambda\subset V$ is a null line subbundle over $S$,
\item $\mathcal{D}$ is a linear metric connection satisfying the Cartan condition, that is $\delta:=\mathcal{D}\mid_{\Lambda} \mod \Lambda$ is an isomorphism from $TS\otimes\Lambda$ to $\Lambda^{\perp}/\Lambda$.
\end{itemize}
In the real Riemannian case we require that $\langle\cdot,\cdot\rangle$ has signature $(n+1,1)$.
A complex conformal Cartan geometry on a complex $n$-manifold is a quadruple $(V,\langle,\rangle,\Lambda,\mathcal{D}^S)$ as above, where all objects are holomorphic and we take $\langle,\rangle$ to be bilinear rather than hermitian.
\end{definition}  
\subsection*{Complexification}
A real structure $\theta$ on a complex manifold $S$ is an anti-holomorphic involution, i.e., an anti-holomorphic map $\theta: \ S\rightarrow S$ such that $\theta^2=\id$. The set of fixed points $S^{\theta}=\{x\in S\ : \ \theta{x}=x\}$ is a real analytic submanifold.

\begin{definition}\label{tools}
A complexification $(S^c,\theta)$ of a real $m$-dimensional manifold $S$ is a complex manifold $S^c$ of complex dimension $m$ together with a real structure $\theta$ on $S^c$ with the fixed point set $(S^c)^{\theta}\simeq S$.
\end{definition}

For any real-analytic manifold $S$ we can construct a complexification $S^c$ of $S$ using holomorphic extensions of real-analytic coordinates on $S$. Indeed, as transition functions are given by real-analytic functions, we can locally holomorphically extend them to obtain a holomorphic atlas and the real structure $\theta$ is given by complex conjugation. Clearly $S\equiv (S^c)^{\theta}$ and we denote $(S^c)^{\theta}$ by $S_{\mathbb{R}}$. Conversely, if $S^c$ is a complexification then by local uniqueness near the submanifold $S$, the transition functions of $S^c$ must coincide with holomorphic extensions of transition functions from $S$. Hence, by local uniqueness of holomorphic extensions, any two complexifications of $S$ are isomorphic in some neighbourhood of $S_{\mathbb{R}}$.

\section{Motivating Example} \label{s2}
In this section, by $\mathbb{CP}^n$ we denote the space $\mathbb{P}(\mathbb{C}^{n+1})$ and by $(\mathbb{CP}^n)^*$ we denote the space $\mathbb{P}((\mathbb{C}^{n+1})^*)$.
We will discuss now a motivating example for the construction of the twistor space of an asymptotically hyperbolic Einstein--Weyl manifold from a 
$2$-manifold equipped with a conformal Cartan connection.

Firstly, we will discuss some geometric properties of the space $\mathbb{CP}^1\times\mathbb{CP}^1$. 

\begin{example}
Let $T\cong(\mathbb{CP}^1)^*\times (\mathbb{CP}^1)^*$ be a quadric in $(\mathbb{CP}^3)^*$.
On $T$ we have two families of lines with normal bundle $\mathcal{O}$ called $\alpha$-lines and $\beta$-lines. $\alpha$-lines do not intersect each other, $\beta$-lines do not intersect each other and any $\alpha$-line intersects any $\beta$-line in exactly one point. Moreover, any point in $T$ belongs to exactly one $\alpha$-line and one $\beta$-line. Thus, we can consider a family of line pairs parametrised by points in $T$. Now, note that the projective tangent space to $T$ at a point $t$ intersects $T$ in the line pair intersecting at the point $t$. Thus, any line pair is determined by the intersection of the tangent space through their intersecting point with $T$. Thus, the moduli space of line pairs of $T$ coincides with the space $T^*\subseteq(\mathbb{CP}^3)$ of tangent spaces to $T$. This is a quadric in $(\mathbb{CP}^3)$. 

The space of deformations of hyperplanes tangent to $T$ is the whole space $\mathbb{CP}^3$ viewed as the space of all hyperplanes in $(\mathbb{CP}^3)^*$. By the adjunction formula the hyperplanes which are non-tangent to $T$ intersect $T$ in curves which have normal bundles isomorphic to $\mathcal{O}(2)$.
\end{example}

\begin{remark}
In the above example we constructed $\mathcal{O}(2)$-curves as 'deformations' of the line pairs. Note that as the line pairs are singular, to be able to obtain non-singular deformations we had to consider non-singular manifolds uniquely determining line pairs. Then we deformed the non-singular manifolds to obtain a non-singular manifold determining a non-singular twistor line. In the general construction we will have a similar situation; to be able to construct non-singular twistor lines we will have to construct some non-singular surfaces determining line pairs.
\end{remark}

Now we will interpret the above correspondence using a conformal Cartan geometry on a quadric $\Sigma$.

Let $\Sigma=\{[z]\in \mathbb{C}\mathbb{P}^3\ | \ z_0z_1+z_2z_3=0\}$ with the following complex Cartan geometry $(V,\langle ,\rangle ,\Lambda,\mathcal{D})$ (see definition $\ref{CartanCon}$):
\begin{itemize}
\item $V=\Sigma\times \mathbb{C}^{4}$ is a rank $4$ trivial vector bundle,
\item $\langle x,y\rangle =x_0y_1+x_1y_0+x_2y_3+x_3y_2$ is an inner product on $\mathbb{C}^{4}$,
\item $\Lambda:=\{([z],w)\in V\ | \ w=\alpha z$ for some $\alpha\in\mathbb{C}\}$ is a rank $1$ null subbundle of $V$,
\item $\mathcal{D}$ is the standard flat connection on the trivial bundle $V$.
\end{itemize}
Note that $\Sigma$ is isomorphic to $\mathbb{CP}^1\times\mathbb{CP}^1$ and the isomorphism $\Upsilon$ is given by
\begin{align*}    
\Upsilon \colon \mathbb{CP}^1\times\mathbb{CP}^1 &\longrightarrow \Sigma\subseteq\mathbb{CP}^3\\
 ([u_0,u_1],[v_0,v_1])&\mapsto [z]=[-u_0v_0,u_1v_1,u_0v_1,u_1v_0]
\end{align*}
and the complex Cartan connection $(V,\langle ,\rangle ,\Lambda,\mathcal{D})$ is a complexification of a (real) Cartan connection on the real $2$-sphere $S^2$.

Observe that for any $[z]=\Upsilon([u_0,u_1],[v_0,v_1])$, the map $\Upsilon$ defines a pair of lines by 
$$\alpha_{[z]}:=\Upsilon([u_0,u_1],[x,y])=\{[-u_0x,u_1y,u_0y,u_1x]\in\Sigma,\ : [x,y]\in\mathbb{CP}^1\},$$
$$\beta_{[z]}:=\Upsilon([x,y],[v_0,v_1])=\{[-xv_0,yv_1,xv_1,yv_0]\in\Sigma,\ : [x,y]\in\mathbb{CP}^1\}.$$
We have that
 $$\Lambda_{[z]}^{0}=\{a\in(\mathbb{C}^{4})^*\ :\ a_1(z_0)+a_0(z_1)+a_3(z_2)+a_2(z_3)=0\},$$ where $\Lambda^{0}\subset V^*$ is the annihilator of $\Lambda$. Furthermore, we have that $\langle\alpha_{[z]},\cdot\rangle\subseteq\mathbb{P}(\Lambda_{[z]}^{0})$ and $\langle\beta_{[z]},\cdot\rangle\subseteq\mathbb{P}(\Lambda_{[z]}^{0})$.
The induced degenerated inner product on $\Lambda^{0}_{[z]}$ defines two null planes $U^+_{[z]}\subseteq\Lambda^{0}$ and $U^-_{[z]}\subseteq\Lambda^{0}$. They are given by the equation 
$$0=\langle a,a\rangle=-a_0a_1+a_2a_3,$$
for $a\in\Lambda^{0}_{[z]}.$ 
Note that the null planes satisfy $\mathbb{P}(U^+_{[z]})=\langle\alpha_{[z]},\cdot\rangle$ and $\mathbb{P}(U^-_{[z]})=\langle\beta_{[z]},\cdot\rangle$.

We define the following fibre bundles over $\Sigma$:
\begin{itemize}
\item $F^+$ to be a bundle of projective lines defined fibrewise by $F^+_{[z]}:=\mathbb{P}(U^+_{[z]})$,
\item $F^-$ to be a bundle of projective lines defined fibrewise by $F^-_{[z]}:=\mathbb{P}(U^-_{[z]})$
\item $F$ to be a bundle of pairs of projective lines defined fibrewise by $F_{[z]}:=F^+_{[z]}\cup F^-_{[z]}.$
\end{itemize}
 As for any $[z]\in \Sigma$ we have that $F_{[z]}\subseteq(\mathbb{CP}^3)^*$, we can consider the space $T=\bigcup_{[z]}F_{[z]}$. Straightforward calculations show that $T$ is a quadric in $(\mathbb{CP}^3)^*$ dual to $\Sigma$.

As a result we have the following diagram
$$ \begin{diagram}
\dgARROWLENGTH=3em
  \node[4]{F\subset\Sigma\times\Sigma^*}\arrow[2]{sw,t}{\pi}\arrow[2]{se,t}{\tilde{\pi}}\\
[2]\node[2]{\Sigma}\node[4]{T}
\end{diagram}. $$
By definition, the fibre of $F$ over $[z]\in\Sigma$ is a pair of lines dual to the line pair through $[z]$, hence $$F=\{([z],\langle [a],\cdot\rangle )\in \Sigma\times\Sigma^*\ : \ [a]\in (\alpha_{[z]}\cup\beta_{[z]})\}.$$
We would like to understand what are fibres of $F$ over $T$.

Observe that on $\Sigma$ we have two families of null curves, namely the $\alpha$ lines and the $\beta$ lines. Denote by $t^+$ the distribution of the tangent bundle defining the foliation by $\alpha$-lines and by $t^-$ the distribution defining the foliation by $\beta$-lines. As we will show in Proposition \ref{lifts}, we can horizontally lift $\alpha$-lines to $F^+$ and $\beta$ lines to $F^-$.

Note that $F^+_{[z]}\cap F^-_{[z]}=\mathbb{P}(\Lambda^*)_{[z]}$ is a point for any $[z]\in \Sigma$. 
\begin{proposition}
$T$ is isomorphic to the leaf space of the foliation of $F$ by horizontal lifts of null curves. The fibres of $F$ over $T$ are pairs of null lines.
\end{proposition}
\begin{proof}
Any horizontal lift of a line pair is uniquely given by the point in which both lines intersect $\mathbb{P}(\Lambda^*)\cong \Sigma^*$. Choose $b\in T$. It is given by the horizontal lift $b_F$ to $F$ of a line pair $(l_1,l_2)\subset\Sigma$ such that $b$ intersect $\mathbb{P}(\Lambda^*)$ over $[z]\in \Sigma$ and $l_1\cap l_2=[z]$. By definition of the horizontal lift, for any $[a]\in l_1\cup l_2$ there exists a point $b_{[a]}\in b_F\in F$. Hence the fibre of $F$ over $b$ is $l_1\cup l_2$. 
\end{proof}
\begin{remark}
Note that as $T=\Sigma^*$ is a quadric in $(\mathbb{CP}^3)^*$, it also admits two families of lines: they are duals to $\alpha$ and $\beta$-lines from $\Sigma$. Now observe that $\tilde{\pi}(b_F)$ from the above proof is the pair of lines in $T$ which intersect in $b$. Those lines are the dual lines to $l_1$ and $l_2$ and hence the fibre of $F$ over $[b]\in T$ is the pair of lines which are dual to the pair of projective lines intersecting in $[b]$.
\end{remark}
\begin{observation}
We have shown that in the above example the situation is symmetrical: $T$ is dual to $\Sigma$ and the fibres of $F$ over $T$ are duals to fibres of $F$ over $\Sigma$. This symmetry does not exist any more if we proceed with the construction for $U\subsetneq \Sigma$. Then the null curves in $\Sigma$ are pieces of lines but the fibres of $F$ over $\Sigma$ are still pairs of projective lines. As a consequence, we obtain that $T$ is a subspace of quadric in $(\mathbb{CP}^3)^*$ consisting of families of projective lines and $U^*\subsetneq T$. This changes the situation from local (a manifold containing pieces of lines) to global (a manifold containing whole projective lines).
\end{observation}
\begin{remark}\label{ident*}
In this example we constructed $F$ as a projective subbundle of $\mathbb{P}(V^*)$. This approach was useful to see the geometry of the situation. However, $\langle\cdot,\cdot\rangle$ induces an isomorphism between $V$ and $V^*$ and hence we have an isomorphism between $\Lambda^{0}\subset V^*$ and $\Lambda^{\perp}\subset V$. Using this, we will simplify the notation in the general construction and will construct $F$ as a fibre subbundle of $\mathbb{P}(V)$.
\end{remark}

\section{Construction of a complex surface containing family of line pairs} \label{s3}

Let $\Sigma$ be a complex surface with a complex Cartan connection $(V,\langle\cdot,\cdot\rangle,\Lambda,\mathcal{D})$ (see Section \ref{s1}) which is a complexification of a real surface $S$ with a Cartan connection.

\subsection*{Null curves and their horizontal lifts}
Now define the following objects:
\begin{itemize}
\item $U^+$ and $U^-$ - the null plane subbundles of $\Lambda^{\perp}\subset V$ defined by the degenerate inner product on $\Lambda^{\perp}$; note that $\Lambda=U^+\cap U^-$,
\item $t^+$ and $t^-$  - the line subbundles of the tangent bundle $T\Sigma$ such that $$\delta (\Lambda\otimes t^+)=U^+/\Lambda\ \text{   and   }\ \delta (\Lambda\otimes t^-)=U^-/\Lambda,$$ where $\delta$ denotes the isomorphism between $T\Sigma\otimes\Lambda$ and $\Lambda^{\perp}/\Lambda$ in the Cartan condition,
\item $C^+$ and $C^-$ - families of curves in $\Sigma$ defined by $t^+$ and $t^-$ (by requiring that $t^{\pm}$ is a tangent space to curves from $C^{\pm}$ families), 
\item $\Sgp$ denotes the leaf space of $C^+$ curves and $\Sgm$ the leaf space of $C^-$ curves,
\item $F^+=\mathbb{P}(U^+)$, $F^-=\mathbb{P}(U^-)$ - fibre bundles over $\Sigma$.
\end{itemize}

\begin{proposition} \label{lifts}
The connection $\mathcal{D}$ satisfies
$\mathcal{D}_X(U^{+})\subseteq U^{+}$ for $X\in t^{+}$ and $\mathcal{D}_Y(U^{-})\subseteq U^{-}$ for $Y\in t^{-}$.
\end{proposition}
\begin{proof}
Let $u\in \Gamma U^+$ and $X\in \Gamma t^{+}$. As $U^+\subseteq\Lambda^{\perp}$ is totally null, for any $\sigma\in \Gamma\Lambda$ we have that
$$\langle \mathcal{D}_Xu,\sigma\rangle=d_X\langle u,\sigma\rangle- \langle u,\mathcal{D}_X\sigma\rangle\in \langle U^+,U^+\rangle=0.$$
Moreover, differentiating the equation $\langle u,u\rangle=0$ we get that
$$\langle \mathcal{D}_Xu,u\rangle =0,$$
thus $$\mathcal{D}_Xu\in (U^+)^{\perp}=U^+.$$
\end{proof}
\begin{wniosek}
$\mathcal{D}$ induces a connection on $F^+$ along curves from $C^+$ and on $F^-$ along curves from $C^-$. Hence we can horizontally lift curves from $C^+$ to $F^+$ and from $C^-$ to $F^-$.    
\end{wniosek}

We restrict $\Sigma$ to an open subset on which curves from families $C^+$ and $C^-$ are simply connected .
\subsection*{Leaf spaces of foliations of horizontal lifts of null curves}
\begin{proposition}
Locally, the horizontal lifts of curves from $C^+$ and $C^-$ to $F^+$ and $F^-$ define foliations of the total spaces of $F^+$ and $F^-$ respectively. Locally, the leaf spaces of the foliations are manifolds. 
\end{proposition}
\begin{proof}
We will prove the proposition for $C^+$ curves. The proof for $C^-$ curves is analogous.
Observe that because $C^+$ is a family of integral curves for $t^+$ distribution, we can restrict $\Sigma$ such that the curves from $C^+$ do not intersect each other. Horizontal lifts of any curve $c^+\in C^+$ are sections of $F^+$ over $c^+$. Hence horizontal lifts of two different curves do not intersect. Moreover, from properties of horizontal lifts, locally two different horizontal lifts of a $c^+$ do not intersect each other.

Hence the horizontal lifts of $C^+$ curves define a foliation of $F^+$ and thus the leaf space of the foliation is (locally in $\Sigma$) a Hausdorff manifold.
\end{proof}

\begin{definition}
Denote by $T^+$ the leaf space of the foliation of $F^+$ by horizontal lifts of curves from $C^+$ and by $T^-$ the leaf space of the foliation of $F^-$ by horizontal lifts from $C^-$. 
\end{definition}

  The lifted curves are transversal to the fibres of $F^+$ and $F^-$ respectively. Thus we can further restrict $\Sigma$ such that any curve from $T^+$ ($T^-$ respectively) intersects any fibre of $F^+$ ($F^-$ respectively) at most once. Moreover, through each point of a particular fibre of $F^+$ ($F^-$ respectively) goes exactly one curve from $T^+$ ($T^-$ respectively). Recall also that all curves intersecting a particular fibre of $F^+$ or $F^-$ arise as horizontal lifts of one curve from the $C^+$ or $C^-$ family respectively. As a consequence we obtain the following proposition.
\begin{proposition} \label{lines}
Any curve $c^+\in C^+$ defines a projective line in $T^+$ given by horizontal lifts of $c^+$. Any curve $c^-\in C^-$ defines a projective line in $T^-$ given by horizontal lifts of $c^-$.\qed
\end{proposition}
As an immediate consequence of the above proposition we obtain that $T^+$ and $T^-$ are foliated by projective lines. 
\begin{definition}\label{plines}
For $c^+\in C^+$ we denote by $l^+_{c^+}\subset T^+$ the projective line given by horizontal lifts of $c^+$. 
For $c^-\in C^-$ we denote by $l^-_{c^-}\subset T^-$ the projective line given by horizontal lifts of $c^-$. 

As any $z\in\Sigma$ defines exactly one curve $c^+$ from $C^+$ and $c^-$ from $C^-$ such that $z\in c^+$ and $z \in c^-$ we set $l^+_z:=l^+_{c^+}$ and $l^-_z:=l^-_{c^-}$.
\end{definition}
 
Note that $\mathbb{P}(\Lambda)$ is a section of both $F^+$ and $F^-$.
 \begin{proposition}\label{trans}
 The elements of $T^+$ and $T^-$ which intersect $\mathbb{P}(\Lambda)$ are transversal to $\mathbb{P}(\Lambda)$.
 \end{proposition}
\begin{proof}
From the Cartan condition we have that $\mathcal{D}|_{\Lambda}$ is an isomorphism between $T\Sigma\otimes\Lambda$ and $\Lambda^{\perp}/\Lambda$. Hence $\mathcal{D}|_{\Lambda}$ is $0$ in $\Lambda^{\perp}/\Lambda$ only for the zero section and consequently the induced connections on $F^+$ and $F^-$ have the property that non-zero elements of $T\Sigma\otimes\mathbb{P}(\Lambda)$ leave $\mathbb{P}(\Lambda)$. In particular the horizontal lifts of curves $c^+\in C^+$ and $c^-\in C^-$ must be transversal to $\mathbb{P}(\Lambda)$
\end{proof} 
 \begin{definition}
 For each $z\in \Sigma$ we have that $F^+_z\cap\F^-_z=\mathbb{P}(\Lambda)_z$ thus we can define a gluing of $T^+$ with $T^-$ by identifying those lifted curves in $F^+$ and $F^-$ which intersect at $\mathbb{P}(\Lambda)_z= F^+_z\cap F^-_z$ for some $z\in \Sigma$:
 $$T:=T^+\bigsqcup_{\sim}  T^-,$$
 where for $t^+\in T^+$ and $t^-\in T^-$ we have $t^+\sim t^-$ if $t^+\cap t^-=\{\mathbb{P}(\Lambda)_z\}$ for some $z\in\Sigma$.
 \end{definition}
 
 Note that as lifted curves are transversal to $\mathbb{P}(\Lambda)$, by possible further restriction of $\Sigma$ we can ensure that any element of $T^+$ is glued with at most one element of $T^-$ and vice versa. Hence, $T$ arises as a gluing of manifolds $T^+$ and $T^-$ on some set. Without loss of generality we can assume that the set is open and connected.

\begin{remark}\label{remrest}
As for any point of $\mathbb{P}(\Lambda)$ we have exactly one element of $T^+$ and one element of $T^-$ which intersect in it, the open subset on which we glue can be identified with $\mathbb{P}(\Lambda)$ and thus with $\Sigma$.
 
By restricting $\Sigma$ to a $\pm$-convex region i.e., to a region in which every curve from $C^+$ intersects every curve from $C^-$ and vice-versa, we get that every projective line from $T^+$ (defined as in Proposition \ref{lines}) intersects every projective line in $T^-$ in exactly one point.
\end{remark}


\begin{proposition}
After restricting $\Sigma$ as in Remark \ref{remrest}, $T$ is a Hausdorff manifold.
\end{proposition}
\begin{proof}
The condition that every line from $T^+$ intersect every line from $T^-$ implies that there are no double points on the boundary of the gluing region and therefore the manifold $T$ is Hausdorff.
\end{proof}

\subsection*{Coordinate description of the gluing}
Now, we will describe the leaf spaces of foliations $T^+$ and $T^-$ and the gluing procedure in another way. We will also introduce coordinates on $T^+$ and $T^-$.

Firstly, recall that for any curve $c^{\pm}\in C^{\pm}$, horizontal lifts of $c^{\pm}$ to $F^{\pm}$ are sections of $F^{\pm}|_{c^{\pm}}$. Choose two intersecting curves $c^+_0\in C^+$ and $c^-_0\in C^-$.
As the curves from $C^+$ and $C^-$ are transversal to each other and as we restricted $\Sigma$ to a $\pm$-convex region, lifts of curves from $C^+$ intersect $F^+|_{c^-_0}$  in exactly one point and lifts of curves from $C^-$ intersect $F^-|_{c^+_0}$ in exactly one point. As a consequence we obtain the following proposition.
\begin{proposition}\label{31}
With possible restriction of $\Sigma$, the manifolds $T^+$ and $T^-$ can be identified with $F^+|_{c^-_0}$ and $F^-|_{c^+_0}$ respectively.\qed
\end{proposition}

The gluing of $F^+|_{c^-_0}$ with $F^-|_{c^+_0}$ induced from the gluing of $T^+$ and $T^-$ can be described using the following coordinates. Let $z,\tilde{z}$ be local coordinates on $\Sigma$ such that curves from $C^+$ are given by $z=\constans$ and those from $C^-$ by $\tilde{z}=\constans$. Suppose moreover, that $c^+_0$ is given by $z=0$ and  $c^-_0$ by $\tilde{z}=0$. This induces local coordinates $(z,w)$ near $\mathbb{P}(\Lambda)$ on $F^+|_{c^-_0}$ and $(\tilde{z},\tilde{w})$ near $\mathbb{P}(\Lambda)$ on $F^-|_{c^+_0}$ in the following way.
\begin{itemize}
\item $(z=a, w=b)$ is a point in a neighbourhood of $\mathbb{P}(\Lambda)_{(z=a,\tilde{z}=0)}$ in $F^+|_{c^-_0}$ which correspond to a horizontal lift of $+$ curves intersecting $\mathbb{P}(\Lambda)$ over the point $(z=a, \tilde{z}=b)$.
\item $(\tilde{z}=\tilde{a}, \tilde{w}=\tilde{b})$ is a point in a neighbourhood of $\mathbb{P}(\Lambda)_{(z=0,\tilde{z}=\tilde{b})}$ in $F^+|_{c^-_0}$ which correspond to a horizontal lift of $-$ curves intersecting $\mathbb{P}(\Lambda)$ over the point $(z=\tilde{a}, \tilde{z}=\tilde{b})$.
\end{itemize}


The following proposition is an immediate consequence of the way we constructed the coordinates.

\begin{proposition}\label{gluingcoordinates}
The gluing $F^+|_{c^-_0}$ and $F^-|_{c^-+_0}$ is given by $z=\tilde{w}$, $w=\tilde{z}$.\qed
\end{proposition}

Recall that projective lines on $T^+$ or $T^-$ are given by horizontal lifts of a fixed $\pm$-curves respectively. This corresponds to fibres of $F^+|_{c^-_0}$ and $F^-|_{c^-+_0}$ respectively. This implies the following corollary.
\begin{wniosek}
The normal bundle of any projective lines in $T$ defined by horizontal lifts of particular $\pm$-curves is isomorphic to $\mathcal{O}$. 
\end{wniosek}

\section{Construction of a twistor space of a self-dual conformal $4$-manifold}\label{s4}
In this section we will show that there is a canonical construction of a $3$-dimensional twistor space $Z$ from the Cartan connection on $\Sigma$ such that $T$ arises as a canonical quotient of $Z$. We will use later the twistor lines to construct $\mathcal{O}(2)$ - curves on $T$. 
\subsection*{Decomposition of $\Lambda$}
In this paragraph we will show that the null line subbundle $\Lambda$ decomposes into a tensor product of pull backs of line bundles on $\Sgp$ and $\Sgm$.
\begin{lemma}\label{LemLambda}
The null line bundle $\Lambda$ decomposes into a tensor product of two line bundles $$\Lambda=\Lambda^+\ot\Lambda^-,$$
such that $\Lambda^+$ is trivial along curves from the family $C+$ and $\Lambda^-$ is trivial along curves from the family $C^-$.
\end{lemma}
\begin{proof}
Using the fact that $\SL(2,\mathbb{C})\times \SL(2,\mathbb{C})$ is a double cover of $\SO(4,\mathbb{C})$ we can find bundles $W$ and $\hat{W}$ such that
$\Wedge^2W$ and $\Wedge^2\hat{W}$ are trivial and
$$V\cong W\ot\hat{W}.$$
We can choose also non-vanishing sections $\epsilon$ of $\Wedge^2W$ and $\hat{\epsilon}$ of $\Wedge^2\hat{W}$ such that $\epsilon\ot\hat{\epsilon}=\langle\cdot,\cdot\rangle.$ Then we have that for $w_1,w_2\in W$ and $\hat{w}_1,\hat{w}_2\in \hat{W}$
$$\langle w_1\ot\hat{w}_1,w_2\ot\hat{w}_2\rangle=\epsilon(w_1,w_2)\hat{\epsilon}(\hat{w}_1,\hat{w}_2).$$
Thus
$$\langle w_1\ot\hat{w}_1,w_1\ot\hat{w}_2\rangle=0 \ \ \ \ \text{and} \ \ \ \  \langle w_1\ot\hat{w}_1,w_2\ot\hat{w}_1\rangle=0,$$
and hence bundles of the form $w\ot\hat{W}$ and $W\ot\hat{w}$ are null subbundles. 
Let $w_1,w_2$ be a basis for $W$ and $\hat{w}_1,\hat{w}_2$ be a basis for $\hat{W}$. Then $w_1\ot\hat{w}_1$, $w_1\ot\hat{w}_2$, $w_2\ot\hat{w}_1$, $w_2\ot\hat{w}_2$ is a basis of $W\ot\hat{W}$ and being a null vector implies that the determinant vanishes. This implies that null vectors are of the form $w\ot\hat{w}$, where $w\in W$ and $\hat{w}\in\hat{W}$.

As a result, the null subbundle $\Lambda\subseteq V$ must be of the form $$\Lambda=\spanu\{w_0\ot\hat{w}_0\},$$
for some $w_0\in W$ and $\hat{w}_0\in\hat{W}$.

Then $$\Lambda^{\perp}=w_0\ot\hat{W}+ W\ot\hat{w}_0.$$
In this setting $U^+=w_0\ot\hat{W}$ and $U^-=W\ot\hat{w}_0$ and 
using the Cartan condition $t^+\cong w_0\ot(\hat{W}/\langle\hat{w}_0\rangle)$ and $t^-\cong (W/\langle w_0\rangle)\ot\hat{w}_0$.

As a consequence $\Lambda^+:=\langle w_0\rangle\subset W$ is constant along $C^+$ curves and $\Lambda^-:=\langle \hat{w}_0\rangle\subset \hat{W}$ is constant along $C^-$ curves, which finishes the proof. 
\end{proof}

\subsection*{Construction of halves of a twistor space}

\begin{definition}
Define 
$$\tilde{U}^+:=U^+\ot(\Lap)^*\ot(\Lap)^*\ \ \ \text{ and }\ \ \ \tilde{U}^-:=U^-\ot(\Lam)^*\ot(\Lam)^*.$$
Observe that $\Lam\ot (\Lap)^*\subset \tilde{U}^+$ and $\Lap\ot (\Lam)^*\subset \tilde{U}^-$.
\end{definition}
From Proposition \ref{lifts} we obtain the following corollary.
\begin{wniosek}
Using the connection $\mathcal{D}$ we can horizontally lift $C^+$ curves to $\tilde{U}^+$ and $C^-$ curves to $\tilde{U}^-$.
\end{wniosek}
Similarly like for lifts of curves to $F$, locally in $\Sigma$, horizontal lifts of $C^+$ curves (and $C^-$ respectively) do not intersect and thus define foliations of $\tilde{U}^+$ and $\tilde{U}^-$.
\begin{definition}
Denote by $V^+$ the leaf space of horizontal lifts of $C^+$ curves to $\tilde{U}^+$. Denote by $V^-$ the leaf space  of horizontal lifts of $C^-$ curves to $\tilde{U}^-$.
\end{definition}
Analogously like for the spaces $T^+$ and $T^-$, we can prove that the spaces $V^+$ and $V^-$ are manifolds.

As $\tilde{U}^+$ arises from $U^+$ by tensoring by a line bundle we get that the projective bundle $\mathbb{P}(\tilde{U}^+)$ is equal to $\mathbb{P}(U^+)$. 
\begin{remark}\label{projU}
We have that
$\mathbb{P}(\tilde{U}^+)=F^+$ and $\mathbb{P}(\tilde{U}^-)=F^-$. 
\end{remark}

\subsection*{Gluing}

By applying the Proposition \ref{trans}, we obtain the following proposition.
\begin{proposition}
The non-zero curves from $V^+$ and $V^-$ are transversal to $\Lam\otimes(\Lap)^*$ and $\Lap\otimes(\Lam)^*$ respectively. \qed
\end{proposition}
As a consequence, locally in $\Sigma$, the non-zero curves from $V^+$  ($V^-$ respectively) intersect $\Lam\otimes(\Lap)^*$ ($\Lap\otimes(\Lam)^*$ respectively) in at most one point. Hence, through any point of $(\Lam\otimes(\Lap)^*)^{\mult}$ and $(\Lap\otimes(\Lam)^*)^{\mult}$ goes exactly one curve from $V^+$ or $V^-$. 

As $\Lap\ot(\Lam)^*=(\Lam\ot(\Lap)^*)^*$, we can define a gluing $\hat{Z}$ of $V^+$ and $V^-$ as follows:
 $$\hat{Z}:=V^+\bigsqcup_{\sim}  V^-,$$
 where for $v^+\in V^+$ and $v^-\in V^-$ we have $v^+\sim v^-$ if $v^+\cap (\Lam\ot(\Lap)^*)=(v^-\cap(\Lap\ot(\Lam)^*))^*$.
 
Unfortunately, $\hat{Z}$ does not need to be Hausdorff. Later we will show how to choose open subsets of $V^+$ and $V^-$ such that the manifold $Z\subset\hat{Z}$ obtained by the gluing of them is Hausdorff.

\subsection*{Coordinate description}
In this paragraph we will discuss how the gluing region of $\hat{Z}$ looks like in more detail. We will also construct local coordinates on the gluing parts of $V^+$ and $V^-$.

Firstly observe that, analogously as in Proposition \ref{31}, we have another description of manifolds $V^+$ and $V^-$.

\begin{proposition}
After possible restriction of $\Sigma$, the manifold $V^+$ can be identified with $\tilde{U}^+|_{c^-_0}$ and the manifold $V^-$ can be identified with $\tilde{U}^-|{c_0^+}$.\qed
\end{proposition}

Recall that in Proposition \ref{31} we identified $T^+$ with $F^+|_{c_0^-}$ and $T^-$ with $F^-|_{c_0^+}$. The gluing region $\mathbb{P}(\Lambda)$ was a union of non-empty open subsets of fibres of $F^+|_{c_0^-}$ and $F^-|_{c_0^+}$ respectively. 
\begin{definition}
Denote by $\tilde{W}^+$ the gluing region in $F^+|_{c_0^-}$ and set $\tilde{W}^+_z:=F^+_z\cap\tilde{W}^+$. Denote by $\tilde{W}^-$ the gluing region in $F^-|_{c_0^+}$ and set $\tilde{W}^-_z:=F^-_{\tilde{z}}\cap\tilde{W}^-.$  
\end{definition}

\begin{definition} \label{taut}
For any vector bundle $\mathcal{A}$ over a manifold $M$ we can define a bundle $\Delta_{\mathcal{A}}$ on the total space of the projective bundle $\mathbb{P}(\mathcal{A})$ by requiring that $(\Delta_{\mathcal{A}})_{l_x}=l_x$, where $l_x$ is a $1$-dimensional subspace in $\mathcal{A}_x$. $\Delta_{\mathcal{A}}$ will be called the tautological bundle along the fibres corresponding to a vector bundle $\mathcal{A}$ and, after restriction to any fibre of $\mathcal{P}$, it is isomorphic to the tautological bundle $\mathcal{O}(-1)$.
\end{definition}

Observe that the bundles $\tilde{U}^+$ and $\tilde{U}^-$ can be naturally identified with the blow-down of $\Delta_{\tilde{U}^+}$ and $\Delta_{\tilde{U}^-}$ - the corresponding tautological bundles along the fibres over $F^+$ and $F^-$ respectively. 
\begin{definition}\label{defW+}
Denote by $W^+\subset \tilde{U}^+$ the image by blow-down of $\Delta_{\tilde{U}^+}|_{\tilde{W}^+}$.
Denote by $W^-\subset \tilde{U}^-$ the image by blow-down of $\Delta_{\tilde{U}^-}|_{\tilde{W}^-}$.
\end{definition}
\begin{observation}
$W^+\setminus\underline{0}$ is the gluing region of $V^+=\tilde{U}^+|_{c^-_0}$. $W^-\setminus\underline{0}$ is the gluing region of $V^-=\tilde{U}^-|_{c^+_0}$.
\end{observation}

Using the properties of blow-down we obtain the following corollary.
\begin{wniosek}\label{wniosekcones}
For any $z\in c_0^-$, $\tilde{z}\in c_0^+$ we have that $W^+_z$ and $W^-_{\tilde{z}}$ are cones in $(\tilde{U}^+)_z$ and $(\tilde{U}^-)_{\tilde{z}}$ with vertex at $0$.
\end{wniosek}

\begin{remark}
By properties of the blow down and of the tautological bundle, $\Delta_{\tilde{U}^+}|_{\tilde{W}^+}\setminus\underline{0}$ is isomorphic with $W^+\setminus \underline{0}$. As a consequence we have a one-to-one correspondence between points of $\Delta_{\tilde{U}^+}|_{\tilde{W}^+}\setminus\underline{0}$ and elements of the gluing part of $V^+$.
\end{remark}

Recall that in Section \ref{s3} we constructed coordinates $(z,w)$ and $(\tilde{w},\tilde{z})$ on gluing regions $\tilde{W}^+$ of $T^+$ and $\tilde{W}^-$ of $T^-$. Using them, we will now construct local coordinates on $W^+$ and $W^-$.

 Let $(z,\tilde{z},\lambda)$ be coordinates on local trivialisation of $\Lam\ot(\Lap)^*$ such that $(z,\tilde{z})$ are coordinates on the base $\Sigma$ and $\lambda$ is a coordinate in the direction of fibres.
 Recall that the coordinates $(z,w)$ on $\tilde{W}^+$ were induced from coordinates $(z,\tilde{z})$ on $\Sigma$ by the intersection point of a curve $s^+\in\tilde{W}^+$ with $\Lambda$. We define coordinates $(z,w,\lambda)$ on $\Delta_{\tilde{U}^+}|_{\tilde{W}^+}\setminus\underline{0}$ in the analogous way (using intersection points of the corresponding lifts of curves from $V^+$ with $\Lam\ot(\Lap)^*$).

Analogously choosing coordinates $(z,\tilde{z},\tilde{\lambda})$ on local trivialisation of $\Lap\ot(\Lam)^*$ we define coordinates $(\tilde{w},\tilde{z},\tilde{\lambda})$ on $\Delta_{\tilde{U}^-}|_{\tilde{W}^-}$. We can require that $\tilde{\lambda}$ is a dual coordinate to $\lambda$ i.e., that if $l\in \Lam\ot(\Lap)^*$ is given in the coordinates $(z,\tilde{z},\lambda)$ by $(a,b,l_0)$ then the point $l^*\in \Lap\ot(\Lam)^*$ is given in coordinates 
 $(z,\tilde{z},\tilde{\lambda})$ by $(a,b,l_0^{-1})$. 
 
From the properties of blow-down we obtain the following proposition.
\begin{proposition}\label{38}
The local coordinates $(z,w,\lambda)$ on $\Delta_{\tilde{U}^+}|_{\tilde{W}^+}$ induce coordinates $(z,\lambda w,\lambda)$ on $(W^+\setminus \underline{0})\subset \tilde{U}^+|_{c^-_0}$ such that they extend (by $\lambda w=0$) continuously through $\underline{0}\subset \tilde{U}^+|_{c^-_0}$.

The local coordinates $(\tilde{w},\tilde{z},\tilde{\lambda})$ on $\Delta_{\tilde{U}^-}|_{\tilde{W}^-}$ induce coordinates $(\tilde{\lambda}\tilde{w},\tilde{z},\tilde{\lambda})$ on $(W^-\setminus \underline{0})\subset \tilde{U}^-|_{c^+_0}$ such that they extend (by $\tilde{\lambda}\tilde{w}=0$) continuously through $\underline{0}\subset \tilde{U}^-|_{c^+_0}$.
\end{proposition}
 \begin{proof}
 Recall that in Corollary \ref{wniosekcones} we showed that $W^+$ is a cone subbundle of $\tilde{U}^+$. $W^+$ is the image of a blow-down of $\Delta_{\tilde{U}^+}|_{\tilde{W}^+}$ and the blow-down map after restriction to $\Delta_{\tilde{U}^+}|_{\tilde{W}^+}\setminus\underline{0}$ is an isomorphism on $W^+\setminus\underline{0}$, thus the coordinates are well defined.
Moreover $\lambda$ tends to $0$ as we approach the zero section and $w$ is bounded (at least locally), which completes the proof.
 \end{proof}
 By construction, we obtain the following coordinate description of the gluing.
 \begin{proposition}\label{43}
 The gluing of $W^+\setminus\underline{0}$ with $W^-\setminus\underline{0}$ is given by $z=\tilde{w}$, $w=\tilde{z}$ and $\lambda=\tilde{\lambda}^{-1}$.\qed
 \end{proposition}

\subsection*{Hausdorffness}
In this paragraph we will show how to choose open subsets of $V^+$ and $V^-$  such that the manifold obtained after gluing is Hausdorff.

The idea is as follows. We choose a tubular neighbourhoods $B^+$ of the zero section in $\tilde{U}^+|_{c^-_0}$ and $B^-$ of the zero section in $\tilde{U}^-|_{c^+_0}$. We can choose $B^+$ and $B^-$ such that the intersection of the image in $\tilde{U}^-|_{c^+_0}$ by the gluing of $B^+\cap (W^+\setminus\underline{0})$ with $B^-$ is empty.

\begin{definition}\label{twist4}
Define $Z^+:=B^+\cup W^+$ and $Z^-:=B^-\cup W^-$ and
$$Z:=Z^+\bigsqcup_{\sim}  Z^-,$$
 where for $v^+\in W^+\setminus\underline{0}$ and $v^-\in W^-\setminus\underline{0}$ we have $v^+\sim v^-$ iff $v^+\cap (\Lam\ot(\Lap)^*)=(v^-\cap(\Lap\ot(\Lam)^*))^*$.
\end{definition}

\begin{proposition}
The space $Z$ defined above is a Hausdorff manifold.
\end{proposition}
\begin{proof}
Clearly $Z$ is a manifold as it arises as a gluing of two manifolds on an open set. It remains to show that $Z$ is Hausdorff. Note that the condition that the intersection of the image in $\tilde{U}^-|_{c^+_0}$ by the gluing of $B^+\cap (W^+\setminus\underline{0})$ with $B^-$ is empty implies that $B^+$ and $B^-$ are open subsets of $Z$ and that on the boundary of the gluing region in $Z$ there are no double points which finishes the proof.
\end{proof}
\begin{wniosek}\label{Pi}
The gluing of $Z^+$ and $Z^-$ is compatible with the gluing of $T^+$ and $T^-$. Hence, using Remark \ref{projU}, we have a projection $\Pi$ from an open subset of $Z$ (the set of points corresponding to $\tilde{U}^+|_{c_0^-}\setminus \underline{0}$ and $\tilde{U}^-|_{c_0^+}\setminus \underline{0}$) to $T$.
\end{wniosek}
\section{Construction of a family of lines with the normal bundle $\mathcal{O}(2)$}\label{s5}
\subsection*{Normal bundle to twistor lines}\label{minitwistors}
\begin{definition}
Let $t_{(a,b)}$ be a projective line in $Z$ which is the gluing of a line in $W^+\subset \tilde{U}^+|_{c^-_0}$ given by $(z,w)=(a,b)$ with a line in $W^-\subset \tilde{U}^-|_{c^+_0}$ given by $(\tilde{w},\tilde{z})=(a,b)$.
\end{definition}

\begin{lemma}
The normal bundle $N_{(a,b)}$ to projective lines $t_{(a,b)}$ decomposes into two line subbundles $N_{(a,b)}=N^+_{(a,b)}\oplus N^-_{(a,b)}$.
\end{lemma}
\begin{proof}
Our aim is to show that the vertical bundle of $\tilde{U}^+$ induces a line subbundle $N^+_{(a,b)}$ of $N_{(a,b)}$ and the vertical bundle of  $\tilde{U}^-$ induces a line subbundle $N^-_{(a,b)}$ of $N_{(a,b)}$.

Firstly, note that at any point $x\in \tilde{U}^+\cap t_{(a,b)}$ the tangent space to $t_{(a,b)}$ at $x$ is contained in the vertical bundle of $\tilde{U}^+|_{c^-_0}$. As the vertical bundle has rank $2$, for $x\in \tilde{U}^+\cap t_{(a,b)}$ it defines a $1$-dimensional subspace $(N^+_{(a,b)})_x$ of $(N_{(a,b)})_x$. The only point of $t_{(a,b)}$ which is not an element of $W^+$ is the point lying on the zero section of $\tilde{U}^-$. Denote it by $\infty$. We will show that the bundle  $(N^+_{(a,b)})$ extends over $\infty$ by $$(N^+_{(a,b)})_{\infty}:=T_{\infty}\underline{\infty},$$
where $\underline{\infty}$ denotes the zero section of $\tilde{U}^-|_{c^+_0}$.

To prove that this really defines a holomorphic line bundle over $t_{(a,b)}$, we 
have to construct a local holomorphic description of $N^+_{(a,b)}$ near $\infty$. Consider a subspace $R_{(a,b)}$ of $V^-$ of those lifted curves that intersect $\Lap\ot(\Lam)^*$ over the curve from $C^+$ given by $z=a$. This is a holomorphic submanifold of $V^-$.
As $t_{(a,b)}\cap \tilde{U}^-|_{c^+_0}$ consists of those curves from $V^-$ that intersect $\Lap\ot(\Lam)^*$ over the point $(z=a, \tilde{z}=b)$ we have that 
$t_{(a,b)}\cap \tilde{U}^-|_{c^+_0}\subset R_{(a,b)}$ and we have that 
$$(N^+_{(a,b)})_x=T_xR_{(a,b)}/T_xt_{(a,b)},$$
for $x\in t_{(a,b)}\cap \tilde{U}^-|_{c^+_0}$. 

Analogously we define a subbundle $N^-_{(a,b)}$ as the vertical bundle of  $U^-_p|_{c^+_0}$ and the tangent space to the zero section in  $U^+_p|_{c^-_0}$. By definition, the bundles $N^+_{(a,b)}$ and $N^-_{(a,b)}$ are transversal to each other which completes the proof of the lemma.

\end{proof}

\begin{proposition}\label{propo11}
The normal bundle to projective lines $t_{(a,b)}$ is $\mathcal{O}(1)\oplus\mathcal{O}(1)$.
\end{proposition}
\begin{proof}

We will show that $N^+_{(a,b)}\cong\mathcal{O}(1)$. The proof that $N^-_{(a,b)}\cong\mathcal{O}(1)$ is analogous.

We will now construct family of curves such that tangent vectors to them at $t=0$ will form sections of $N^+$.

In the $\tilde{U}^+$ part, using the coordinates $(z,\lambda w, \lambda)$, we define 
$$\gamma_{\lambda}(t)=(a, \lambda (b+t),\lambda,),$$

Those curves glue to curves in the $\tilde{U}^-$ part given in the coordinates $(\tilde{\lambda}\tilde{w},\tilde{z},\tilde{\lambda})$ by 
 $$\gamma_{\tilde{\lambda}}(t)= (\tilde{\lambda} a,b+t,\tilde{\lambda}).$$

A direct calculation of derivatives at $t=0$ gives $\frac{\de}{\de t}\gamma_{\lambda}|_{t=0}=(0,\lambda,0)$ and $\frac{\de}{\de t}\gamma_{\tilde{\lambda}}|_{t=0}=(0,1,0)$.

The local coordinates $(z,\lambda w, \lambda)$ and$(\tilde{\lambda}\tilde{w},\tilde{z},\tilde{\lambda})$ are defined in a neighbourhood of $t_{(a,b)}$ everywhere except from points $\lambda=0$ and $\tilde{\lambda}=0$ hence the family of curves defined above define a holomorphic section, say $u$, of $N^+_{(a,b)}$ for all points of $t_{(a,b)}$ except from $\lambda=0$ and $\tilde{\lambda}=0$.

For points $\tilde{\lambda}=0$, the coordinates in the direction of $N^+$ (i.e., $\tilde{z}$) are also defined in some neighbourhood of $\tilde{\lambda}=0$. Hence the formula for $u$ in coordinates is also valid there and hence $u$ extends holomorphically over $\tilde{\lambda}=0$.

Observe that, as coordinates$(z,\lambda w, \lambda)$ extend continuously (by zero) through $\underline{0}$ (see Proposition \ref{38}), the section $u$ tends to zero when $\lambda$ tends to $0$. By Riemann's Theorem on Removable Singularities, $u$ extends holomorphically through $\lambda=0$ and $u(\lambda=0)=0$. 

Hence we constructed a holomorphic section of $N^+_{(a,b)}$ vanishing at exactly one point. By the Birkhoff–-Grothendieck theorem (\cite{Grot}) $N^+_{(a,b)}\cong\mathcal{O}(1)$.

\end{proof}
\subsection*{Real structure}
Recall that the Cartan connection on $\Sigma$ is a complexification of a Cartan connection on $S$. Let $\theta$ be a real structure on $\Sigma$ coming from complexification (see Definition \ref{tools}). As the Cartan connection $(V,\langle\cdot,\cdot\rangle,\Lambda,\mathcal{D})$ comes from complexification, $\theta$ interchanges $t^+$ and $t^-$ and thus $C^+$ and $C^-$ curves. Moreover, the real structure $\theta$ on $V$ defines also an anti-holomorphic isomorphism between $F^+$ and $F^-$ which induces an anti-holomorphic isomorphism between $T^+$ and $T^-$. For simplification the anti-holomorphic isomorphisms defined above will be also denoted by $\theta$. 

\begin{remark}
Recall that $\Sigma_{\mathbb{R}}$ denotes the real submanifold of $\Sigma$ coming from the complexification. Then for any $z\in\Sigma_{\mathbb{R}}$, the curves $s^+\in T^+$ and $s^-\in T^-$ that intersect $\mathbb{P}(\Lambda)$ over $z$ are interchanged by $\theta$.
\end{remark} 
We can ensure that $T^+$ and $T^-$ are such that the manifold $T$ admits a real structure canonically induced by the complexification of the Cartan geometry. The real structure interchanges projective lines from $T^+$ with projective lines from $T^-$ and on $\Sigma\subset T$ it gives back the real structure $\theta$ from complexification. To simplify the notation we denote the real structure on $T$ by $\theta$.

Analogously, the real structure $\theta$ induces an anti-holomorphic isomorphism $\theta$ between $V^+$ and $V^-$. We define an anti-holomorphic isomorphism between $V^+$ and $V^-$ by $$\theta_Z:=(-\id)\circ\theta,$$
and we can ensure that $V^+$ and $V^-$ are such that $\theta_Z$ is a real structure on $Z$. Observe that for $z$ not in the zero section $\mathbb{P}(\theta_Z(z))=\theta(\mathbb{P}(z))$.

In coordinates $$\theta_Z(z,\lambda w,\lambda)=(-\overline{\lambda}\overline{w},\overline{z}, -\overline{\lambda})$$
 and 
 $$\theta_Z(\tilde{\lambda}\tilde{w},\tilde{z},\tilde{\lambda})=(\overline{\tilde{z}},-\overline{\tilde{\lambda}}\overline{\tilde{w}}, -\overline{\tilde{\lambda}}).$$
\begin{proposition}\label{realCartan5}
The real structure $\theta_Z$ induces on $t_{(a,\bar{a})}$ the antipodal map.
\end{proposition}
\begin{proof}
Firstly observe that the real structure interchange the zero in $\tilde{U}^+|_{a}$ with the zero in $\tilde{U}^-|_{\overline{a}}$. Any other point on $t_{(a,\bar{a})}$ can be written using the coordinates $(z,\lambda w,\lambda)$ as $(a,l\overline{a},l)$ for some $l\in\mathbb{C}^{\mult}$.  We have $$\theta_Z(a,l\overline{a},l)=(-\overline{l}a,\overline{a},-\overline{l}),$$
where $(-\overline{l}a,\overline{a},-\overline{l})$ is written using the coordinates $(\tilde{\lambda}\tilde{w},\tilde{z},\tilde{\lambda})$. This corresponds (by gluing) in the coordinates $(z,\lambda w,\lambda)$ to a point
$(a,-\overline{l}^{-1}\overline{a},-\overline{l}^{-1})\in t_{(a,\bar{a})}$ hence the real structure preserves $t_{(a,\bar{a})}$. As the equation $$l=-\overline{l}^{-1}$$ does not have solutions in the complex numbers, the induced real structure on $t_{(a,\bar{a})}$ is antipodal.
\end{proof}

\subsection*{Main theorem}
\begin{theorem}
$Z$ is a twistor space of a self-dual conformal $4$-manifold.
\end{theorem}
\begin{proof}
We have shown that $Z$ is a complex $3$-manifold with the real structure $\theta_Z$. Moreover, by Proposition \ref{realCartan5}, $t_{(a,\bar{a})}$ are $\theta_Z$-invariant projective lines on $Z$ such that the induced real structure on  $t_{(a,\bar{a})}$ is antipodal. In Proposition \ref{propo11} we have also shown that normal bundle to  $t_{(a,\bar{a})}$ is $\mathcal{O}(1)\oplus\mathcal{O}(1)$ which, by Theorem \ref{TPenrose}, completes the proof.
\end{proof}
\subsection*{Minitwistor lines}
\begin{definition}
Denote by $M$ the self-dual conformal $4$-manifold such that the twistor space of $M$ is $Z$.
\end{definition}
From Kodaira theorem \cite{KK}, there exists a $4$-dimensional complex manifold $M^c$ which is a local moduli space of curves (twistor lines) with the normal bundle $\mathcal{O}(1)\oplus\mathcal{O}(1)$ which comes as deformations of curves $t_{(a,b)}$. 

\begin{remark}\label{transrem}
The fact that the normal bundle to twistor lines is isomorphic to $\mathcal{O}(1)\oplus\mathcal{O}(1)$ implies that the nearby twistor lines intersect each other in at most one point. In particular any nearby twistor lines from the moduli space intersect any $t_{(a,b)}$ in at most one point.
\end{remark}
Recall that $T$ is a complex surface defined in Section \ref{s3} and that in order to show that $T$ is a minitwistor space, we have to prove that $T$ contains a projective line with normal bundle $\mathcal{O}(2)$.
\begin{proposition}
The projection $\Pi$ (see Corollary \ref{Pi}) from an open subset of $Z$ to $T$ is given by a quotient of $Z$ by a local $\mathbb{C}^{\mult}$ action $\circ$.
\end{proposition}
\begin{proof}
 To see this, note that the projection on the gluing part acts by collapsing projective lines $t_{(a,b)}$ (given in each of the gluing parts by a $1$-dimensional vector spaces) to a point and on parts where we do not have a gluing, it acts by projectivisations of $U^+_p|_{c_0^-}$ and of $\tilde{U}^-|_{c_0^+}$. Define a local $\mathbb{C}^{\mult}$ action $\circ$ as the action which is the scalar multiplication in the fibres of $\tilde{U}^+|_{c^-_0}$ and inverse of the scalar multiplication of  $\tilde{U}^+|_{c^-_0}$. By Proposition \ref{43}, the definition of the $\circ$ is compatible with the gluing, thus the local $\mathbb{C}^{\mult}$ action $\circ$ on $Z$ is well-defined.
 The set of fixed points are the zero sections in $Z^+$ and $Z^-$ and the projection can be described as a gluing of fibrewise projectivisations of $Z^+$ and $Z^-$. 
\end{proof} 
\begin{proposition}
By construction, $\circ$ corresponds to an $S^1$ action on $M$. The $S^1$ action is semi-free, for some neighbourhood in $M$ of the points corresponding to twistor lines $t_{(a,b)}$.
\end{proposition}
\begin{proof}
Let $l\in \mathbb{C}^{\mult}$ and $u$ be any twistor line in $Z$. As $\circ$ is invertible, $l\circ u$ is a projective line in $Z$ and it is immediate to write an explicit isomorphism between normal bundles to $u$ and $l\circ u$. Moreover, if $|l|=1$ and $u$ is a real twistor line, we get that $l\circ u$ is a real twistor line. 

Note that all $t_{(a,b)}$ are invariant under $\circ$ hence the corresponding points in $M$ are fixed under the $S^1$ action. If $u$ intersects some $t_{(a,b)}$, then by Remark \ref{transrem}, $u\cap t_{(a,b)}=\{z\}$ is a singleton. For $z$ not in the zero sections of $Z^+$ and $Z^-$ we have that $S^1\circ z\cong S^1$. This implies that the orbit of $u$ is isomorphic to $S^1$. 
\end{proof}
 Remark \ref{transrem} implies that there exists a real twistor line $t$ from $M^c$ which is transversal to $\circ$. Thus, the vector field which is a generator of the action $\circ$ induces a non-vanishing section of the normal bundle of $t$ and thus the rank $1$ trivial subbundle $N^s(t)$ of the normal bundle $N(t) $ of $t$.
 This can be summarised in the following corollary.
\begin{wniosek}
There exists a real twistor line $t$ which is transversal to the action $\circ$. The normal bundle to such a twistor line admits a trivial line subbundle $N^s(t)$.
\end{wniosek} 
\begin{lemma}\label{lemmaminitwist}
For any twistor line $t$ which is transversal to the action $\circ$, we have that $\Pi(t)$ is a projective line in $T$ with the normal bundle $\mathcal{O}(2)$. Hence $\Pi(t)$ is a minitwistor line. Moreover, if $t$ is a real twistor line, then $\Pi(t)$ is a real minitwistor line and the induced real structure on $\Pi(t)$ is antipodal.
\end{lemma}
\begin{proof}

  The quotient bundle of $N$ by $N^s$ is the normal bundle $N(\Pi(t))$ to the projective line $\Pi(t)\subset T$. From this, we obtain the following exact sequence:

$$0\longrightarrow N^s \longrightarrow N(t)\longrightarrow N(\Pi(t))\longrightarrow 0.$$
As a consequence, the degrees of the bundles satisfy:

$$\deg N^s+\deg N(\Pi(t))=\deg N(t),$$
thus $\deg N(\Pi(t))=2$ and by the Birkhoff--Grothendieck theorem (\cite{Grot}) $$N(\Pi(t))\cong \mathcal{O}(2).$$

By direct application of Proposition \ref{realCartan5}, we obtain that the real structure on $\Pi(t)$ is antipodal.

\end{proof}
Now we are in position to state the main theorem of the paper.
\begin{theorem}
$T$ is a minitwistor space of an Einstein--Weyl manifold.
\end{theorem}
\begin{proof}
$T$ is a complex surface with a real structure $\theta$. By Lemma \ref{lemmaminitwist}, $T$ admits a projective curve $t$ which is $\theta$-invariant and the induced real structure on $t$ is antipodal. Moreover the normal bundle to $t$ is $\mathcal{O}(2)$ thus by Theorem \ref{THitchin}, $T$ is a minitwistor space of an Einstein--Weyl $3$-manifold.
\end{proof}

\begin{observation}
The projective lines with normal bundle $\mathcal{O}(1)$ arise as quotients of images in $T$ of (possibly restricted) fibres of bundles $U^+$ and $U^-$. The pair of projective lines given by fibres $z=a$ and $\tilde{z}=b$ meet in the point given by $\Pi(t_{(a,b)})$. As those the twistor lines which (by quotients) give minitwistor lines arise as deformations of $t_{(a,b)}$ we can view the line pairs as degenerations of the minitwistor lines.
\end{observation}
Following the result of LeBrun \cite{Le} (see Lemma \ref{lemmaLe}) we obtain the following corollary.
\begin{wniosek}
The Einstein--Weyl manifold obtained from the minitwistor space $T$ is asymptotically hyperbolic.
\end{wniosek}

\section{Summary}

We have shown that there is a natural construction of the minitwistor space of an asymptotically hyperbolic Einstein--Weyl manifold $B$ from a real-analytic surface $\Sigma$ equipped with a real-analytic conformal Cartan connection. The minitwistor space contains a family of pairs of projective lines which are deformations of the minitwistor lines and encode the asymptotic end of the Einstein--Weyl manifold. 
Moreover, we have shown that there exists a natural construction of a twistor space $Z$ of the self-dual conformal $4$-manifold $M$ from $\Sigma$ such that $S$ is a submanifold of $M$ and $B$ arises as quotient of $M$ by an $S^1$ action. As a consequence, the asymptotically hyperbolic Einstein--Weyl manifold $B$, constructed from $\Sigma$, admits a distinguished Gauduchon gauge (see \cite{JT}).

The constructed twistor space is the result of a special case of the Generalized Feix--Kaledin construction of quaternionic manifolds obtained recently by D. Calderbank and the author (\cite{BC}, compare \cite{Feix2}, \cite{KAL1}). In \cite{BC} we construct a family of twistor spaces from which $T$ arises as quotient and further investigate the obtained Gauduchon gauges.
\section{Acknowledgements}
The results presented in this paper are part of author's PhD thesis. The author would like to thank Professor David Calderbank for his supervision during the project, suggesting research directions and encouragement.

The research was supported by University of Bath PhD studentship and Department of Mathematics and Computer Science of Jagiellonian University in Krakow.
The author would also like to thank Institute of Mathematics of Jagiellonian University in Krakow for financial support.

\bibliographystyle{amsalpha}

\providecommand{\bysame}{\leavevmode\hbox to3em{\hrulefill}\thinspace}
\providecommand{\MR}{\relax\ifhmode\unskip\space\fi MR }
\providecommand{\MRhref}[2]{%
  \href{http://www.ams.org/mathscinet-getitem?mr=#1}{#2}
}
\providecommand{\href}[2]{#2}

\end{document}